\documentclass[11pt]{amsart}
\usepackage{verbatim}
\usepackage{rotating} 
\usepackage{amssymb}
\usepackage[bookmarks,colorlinks,breaklinks]{hyperref}  
\hypersetup{linkcolor=blue,citecolor=blue,filecolor=blue,urlcolor=blue}
\usepackage{mathrsfs,amsfonts,amsmath,amssymb,epsfig,amscd,xy,amsthm} 

\newtheorem{theorem}{Theorem}[section]
\newtheorem{proposition}[theorem]{Proposition}

\newtheorem{corollary}[theorem]{Corollary}

\newtheorem{definition}[theorem]{Definition}

\newtheorem{conjecture}[theorem]{Conjecture}

\theoremstyle{plain}

\theoremstyle{remark}

\newtheorem{remark}[theorem]{Remark}

\newtheorem{remarks}[theorem]{Remarks}

\newcommand{\C}{{\mathbb C}}
\newcommand{\F}{{\mathbb F}}
\newcommand{\G}{{\mathbb G}}
\newcommand{\Q}{{\mathbb Q}}

\newcommand{\Z}{{\mathbb Z}}
\newcommand{\N}{{\mathbb N}}

\newcommand{\Gm}{\mathbb{G}_{\text{m}}}

\newcommand{\x}{{\bf x}}
\renewcommand{\u}{{\bf u}}

\DeclareMathOperator{\Id}{Id}

\DeclareMathOperator{\End}{End}

\newcommand{\bG}{{\mathbb G}}

\newcommand{\bF}{{\mathbb F}}
\newcommand{\Fq}{\bF_q}
\newcommand{\lra}{\longrightarrow}
\newcommand{\cO}{\mathcal{O}}

\author{P. Corvaja}
\address{Pietro Corvaja\\
Dipartimento di Scienze Matematiche\\
Informatiche e Fisiche\\ 
Universit\`a di Udine\\
via delle Scienze, 206\\
33100 Udine\\
Italy}
\email{pietro.corvaja@uniud.it}

\author{D. Ghioca}
\address{
Dragos Ghioca\\
Department of Mathematics\\
University of British Columbia\\
1984 Mathematics Road\\
Vancouver, BC V6T 1Z2\\
Canada
}
\email{dghioca@math.ubc.ca}

\author{T. Scanlon}
\address{Thomas Scanlon\\
University of California, Berkeley\\
Mathematics Department\\
Evans Hall\\
Berkeley, CA 94720-3840\\
USA}
\email{scanlon@math.berkeley.edu}

\author{U. Zannier}
\address{Umberto Zannier\\
Scuola Normale Superiore\\
Piazza dei Cavalieri, 7\\
56126 Pisa\\
Italy}
\email{u.zannier@sns.it}

\thanks{The second author has been partially supported by a Discovery Grant from
the National Science and Engineering Board of Canada.  
The third author has been partially supported by 
grant DMS-1363372 of the United States National Science
Foundation and a Simons Foundation Fellowship.}


\begin{document}
	\title[The dynamical Mordell-Lang problem]{The dynamical Mordell-Lang conjecture for endomorphisms of semiabelian varieties defined over fields of  positive characteristic}

\begin{abstract}
Let $K$ be an algebraically closed field of prime characteristic $p$, let $X$ be a semiabelian variety defined over a finite subfield of $K$, let $\Phi:X\lra X$ be a regular self-map defined over $K$, let $V\subset X$ be a subvariety defined over $K$, and let $\alpha\in X(K)$. The Dynamical Mordell-Lang Conjecture in characteristic $p$ predicts that the set $S=\{n\in\N\colon \Phi^n(\alpha)\in V\}$ is a union of finitely many arithmetic progressions, along with finitely many $p$-sets, which are sets of the form $\left\{\sum_{i=1}^m c_i p^{k_in_i}\colon n_i\in\N\right\}$ for some $m\in\N$, some rational numbers $c_i$ and some non-negative integers $k_i$. We prove that this conjecture is equivalent with some difficult diophantine problem in characteristic $0$. In the case $X$ is an algebraic torus, we can prove the conjecture in two cases: either when $\dim(V)\le 2$, or when no iterate of $\Phi$ is a group endomorphism which induces the action of a power of the Frobenius on a positive dimensional algebraic subgroup of $X$. 
\end{abstract}

	\maketitle


 \section{Introduction}
\label{sec:introduction}

In this paper, as a matter of convention, any subvariety of a given variety is assumed to be closed.  
We denote by $\N$ the set of positive integers, we let 
$\N_0:=\N\cup\{0\}$, and let $p$ be a prime number. 
An arithmetic progression is a set of the form
$\{mk+\ell:\ k\in\N_0\}$  
for some $m,\ell\in\N_0$; note that when $m=0$, this set
is a singleton. 
For a set $X$ endowed with a self-map $\Phi$, and for $m\in\N_0$, we let $\Phi^m$
denote the $m$-th iterate $\Phi\circ\cdots\circ \Phi$, where $\Phi^0$
denotes the identity map on $X$. If $\alpha\in X$, we define its orbit under $\Phi$ as 
$\cO_\Phi(\alpha):=\{\Phi^n(\alpha)\colon n\in\N_0\}$.  

Motivated by the classical Mordell-Lang conjecture proved by Faltings \cite{Fal94} (for abelian varieties) and Vojta \cite{Voj96} (for semiabelian varieties), the Dynamical Mordell-Lang Conjecture predicts that for a quasiprojective variety $X$ endowed with a (regular) self-map $\Phi$ defined over a field $K$ of characteristic $0$, given a point 
$\alpha\in X(K)$ and a subvariety $V$ of $X$, the set $$S:=\{n\in\N_0\colon \Phi^n(\alpha)\in V(K)\}$$ 
is a finite union of arithmetic progressions (see \cite[Conjecture~1.7]{GT-JNT} and also, for a thorough discussion of the Dynamical Mordell-Lang Conjecture, see the book \cite{book}).  

With the above notation for $X,\Phi,K,V,\alpha,S$, if $K$ has characteristic $p$  then $S$ may be infinite without containing an infinite arithmetic progression (see \cite[Example~3.4.5.1]{book} and \cite[Example~1.2]{G-TAMS}). Similar to the classical Mordell-Lang conjecture in characteristic $p$, one has to take into account varieties defined over finite fields (see the proof of the classical Mordell-Lang conjecture in positive characteristic by Hrushovski \cite{Hrushovski}). So, motivated by the structure theorem of Moosa-Scanlon \cite{Moosa-Scanlon} describing in terms of $F$-sets the intersection of a subvariety of a semiabelian variety (defined over a finite field) with a finitely generated group (see also \cite{groups}), the following conjecture was advanced.   

\begin{conjecture}[Dynamical Mordell-Lang Conjecture in positive characteristic]
\label{char p DML}
Let $X$ be a quasiprojective variety defined over a field $K$ of characteristic $p$. Let $\alpha\in X(K)$, let $V\subseteq X$ be a subvariety defined over $K$, and let $\Phi:X\lra X$ be an endomorphism defined over $K$. Then the set $S:=S(X,\Phi,V,\alpha)$ consisting of all $n\in\N_0$ such that $\Phi^n(\alpha)\in V(K)$ is a union of finitely many arithmetic progressions along with finitely many sets of the form
\begin{equation}
\label{form of the solutions}
\left\{\sum_{j=1}^{m} c_j p^{k_j n_j}\text{ : }n_j\in\N_0\text{ for each }j=1,\dots m\right\},
\end{equation}
for some $m\in\N$, some $c_j\in\Q$, and some $k_j\in\N_0$. 
\end{conjecture}
A set as in \eqref{form of the solutions} is called a \emph{$p$-set}.

Now, clearly, the set $S$ may contain infinite arithmetic progressions if $V$ contains some periodic subvariety under the action of $\Phi$ which intersects $\cO_\Phi(\alpha)$. As an aside, we note that if $X$ is a semiabelian variety (see Section~\ref{subsection semiabelian}) and $\Phi$ is a group endomorphism of $X$, then the periodic subvarieties of $V\subset X$ under the action of $\Phi$ are
necessarily translates of subgroups contained in $V$; such translates can
be classified and the maximal subgroups corresponding to them are finite in
number (for more details, see \cite{BG06, Zannier}). On the other hand, it is possible for $S$ to contain nontrivial sets of the form \eqref{form of the solutions} (where $m$ may be larger than $1$ and also where the $c_i$'s may not be integers, as shown by \cite[Examples~1.2~and~1.4]{G-TAMS}).  It is likely that all cases when the set of solutions $S$ is infinite, but it is not a finite union of arithmetic progressions occur in the presence of an algebraic group; however, proving this assertion is probably just as difficult as the original Conjecture~\ref{char p DML}.

So far, Conjecture~\ref{char p DML} is known to hold in the following two cases:
\begin{enumerate}
\item[(i)] if $X$ is a semiabelian variety defined over a finite field and $\Phi$ is an algebraic group endomorphism whose action on the tangent space at the identity is given through a diagonalizable matrix (see \cite[Proposition~13.3.0.2]{book}). Actually, \cite[Proposition~13.3.0.2]{book} was stated only in the special case $X=\bG_m^N$, but the proof carries over almost verbatim for the general case of semiabelian varieties defined over a finite field. Furthermore, in this case, the set $S$ from the conclusion of Conjecture~\ref{char p DML} consists only of finitely many arithmetic progressions, i.e., the more complicated $p$-sets from \eqref{form of the solutions} do not appear.
\item[(ii)] if $X=\bG_m^N$ and $V\subset X$ is a curve (see \cite[Theorem~1.3]{G-TAMS}). In this case, the $p$-sets from \eqref{form of the solutions} may appear, but they consist of only one nontrivial power of the prime $p$ (i.e., $m\le 2$ and $k_1=0$ with the notation as in \eqref{form of the solutions}). 
\end{enumerate}

In this paper we prove two new results towards Conjecture~\ref{char p DML} by showing it holds for any surface contained in $\bG_m^N$ and for any regular self-map $\Phi:\bG_m^N\lra \bG_m^N$ (not necessarily a group endomorphism), and also prove that it holds for any subvariety $V\subset \bG_m^N$ assuming $\Phi$ is an algebraic group endomorphism with the property that no iterate of it restricts to being a power of the Frobenius on a proper algebraic subgroup of $\bG_m^N$. 

\begin{theorem}
\label{surface result}
Let $\Phi:\bG_m^N\lra \bG_m^N$ be a regular self-map defined over an algebraically closed field $K$ of characteristic $p$, let $V\subset \bG_m^N$ be a variety of dimension at most equal to $2$, and let $\alpha\in \bG_m^N(K)$. Then the set of $n\in\N_0$ such that $\Phi^n(\alpha)\in V(K)$ is a finite union of arithmetic progressions along with finitely many sets of the form
\begin{equation}
\label{form of the solutions surface}
\left\{c_0+c_1p^{k_1n_1}+c_2p^{k_2n_2}\colon n_1,n_2\in\N_0\right\},
\end{equation}
for some $c_0,c_1,c_2\in\Q$ and some $k_1,k_2\in\N_0$.
\end{theorem}

We remark that Theorem~\ref{surface result} in the case $V$ is a surface formally implies the conclusion of Theorem~\ref{surface result} in the case of curves; hence, it suffices to assume $V$ is an irreducible surface in the hypotheses of Theorem~\ref{surface result}.

As shown in \cite[Example~1.2]{G-TAMS} (which provides the example of a hypersurface $V\subset \bG_m^3$), there are examples of surfaces $V\subset \bG_m^N$ whose intersection with an orbit of a regular self-map of $\bG_m^N$ yields a set as in \eqref{form of the solutions surface} where $c_1,c_2,k_1,k_2$ are all nonzero.

\begin{theorem}
\label{Frobenius result}
Let $\Phi:\bG_m^N\lra \bG_m^N$ be an algebraic group endomorphism, let $V\subset \bG_m^N$ be a subvariety defined over an algebraically closed field of characteristic $p$, and let $\alpha\in \bG_m^N(K)$. Assume there is no nontrivial connected algebraic subgroup $G$ of $\bG_m^N$ such that an iterate of $\Phi$ induces an endomorphism of $G$, which equals a power of the usual Frobenius. Then the set of $n\in\N_0$ such that $\Phi^n(\alpha)\in V(K)$ is a finite union of arithmetic progressions along with finitely many sets of the form
$$\left\{\sum_{j=1}^{m} c_j p^{k_j n_j}\text{ : }n_j\in\N_0\text{ for each }j=1,\dots m\right\},$$
for some $m\in\N$, some $c_j\in\Q$, and some $k_j\in\N_0$. 
\end{theorem}

Both Theorems~\ref{surface result} and \ref{Frobenius result} are proved by showing that in these cases, Conjecture~\ref{char p DML} reduces to some questions regarding polynomial-exponential equations, which can be solved with the known Diophantine tools. More generally, we prove that Conjecture~\ref{char p DML} in the more general case of an arbitrary group endomorphism of $\bG_m^N$ reduces to some deep Diophantine questions, whose solution is well-beyond the presently known techniques.

\begin{theorem}
\label{theorem reduction to polynomial-exponential equations}
Let $\{u_k\}$ be a  linear recurrence sequence of integers, let $m, c_1,\dots, c_m\in\N$, and let $p$ be a prime number with 
$\sum_{i=1}^m c_i < p -1$. Then there exists $N\in\N$, there exists an algebraically closed field $K$, there exists an algebraic group endomorphism $\Phi:\bG_m^N\lra \bG_m^N$, there exists $\alpha\in \bG_m^N(K)$ and there exists a subvariety $V\subset \bG_m^N(K)$ such that the set of all $n\in\N_0$ for which $\Phi^n(\alpha)\in V(K)$ is precisely the set of all $n\in\N_0$ such that 
\begin{equation}
\label{polynomial-exponential equation}
u_n=\sum_{i=1}^m c_i p^{n_i},
\end{equation}
for some $n_1,\dots, n_m\in\N_0$.
\end{theorem}

The polynomial-exponential equation \eqref{polynomial-exponential equation} is very deep and despite an extensive research being done on this problem, the case of a general linear recurrence sequence $\{u_n\}$ coupled with $m>2$ (i.e., at least three powers of the prime $p$ in \eqref{polynomial-exponential equation}) is  beyond the known methods (for more details, see \cite{CZ1, CZ3}). We also note that applying a similar construction as in the proof of Theorem~\ref{theorem reduction to polynomial-exponential equations}, one can establish a similar result in the context of arbitrary semiabelian varieties $X$ defined over a finite field; in this case,  the right hand-side of \eqref{polynomial-exponential equation} is replaced by some linear combination of powers of the roots of the minimal polynomial in $\Z[x]$ which kills the Frobenius acting on $X$ (in the case  $X=\bG_m^N$, the minimal polynomial in $\Z[x]$ killing the Frobenius is simply $x-p$; for more details, see Section~\ref{subsection semiabelian}, especially equation \eqref{general equation endomorphism}).

Finally, we note that since already the case of semiabelian varieties is very difficult, it is hard to predict how one can approach Conjecture~\ref{char p DML} for other ambient quasiprojective varieties $X$; also, many of the $p$-adic tools used in proving special cases of the Dynamical Mordell-Lang Conjecture in characteristic $0$ (see \cite[Chapter~6]{book}) do not work in positive characteristic.

We discuss now the strategy for proving our results. We have two main ingredients for the proof of Theorems~\ref{surface result}~and~\ref{Frobenius result}: on one hand, we have Theorem~\ref{Moosa-Scanlon theorem} (see also the more precise statement from Corollary~\ref{precise intersection surface} in the case of intersections with surfaces), and on the other hand, we have various  Diophantine results regarding solutions to polynomials-exponential equations (see Theorem~\ref{theorem polynomial-exponential}).  With the notation as in Theorems~\ref{surface result}~and~\ref{Frobenius result}, there exists a finitely generated group $\Gamma$ such that the orbit of $\alpha$ under $\Phi$ is contained inside $\Gamma$, and so, we obtain $V(K)\cap\cO_\Phi(\alpha)$ by intersecting first $V$ with $\Gamma$ and then with $\cO_{\Phi}(\alpha)$. As we will prove in  Theorem~\ref{theorem reduction} (see also Theorem~\ref{theorem reduction to polynomial-exponential equations} for the converse statement), solving Conjecture~\ref{char p DML} for regular self-maps on $\bG_m^N$ reduces to solving polynomial-exponential equations of the form \eqref{polynomial-exponential equation}. If $V$ is a surface, we use Corollary~\ref{precise intersection surface} to reduce to the case there are at most $2$ powers of $p$ in the right-hand side from \eqref{polynomial-exponential equation}; this allows us to completely solve the case of surfaces contained in a torus. Similarly, the hypothesis from Theorem~\ref{Frobenius result} on the endomorphism $\Phi$ of $\bG_m^N$ reduces the polynomial-exponential equation from \eqref{polynomial-exponential equation} to a Diophantine question, which we know how to answer. Apart from these two cases, the polynomial-exponential equations of type \eqref{polynomial-exponential equation} are beyond the reach of the known Diophantine methods. 

We sketch below the plan of our paper. In Section~\ref{subsection F-sets} we discuss the classical Mordell-Lang problem for semiabelian varieties in characteristic $p$, by stating the results of Moosa-Scanlon \cite{Moosa-Scanlon}  and then  explaining their connections to Conjecture~\ref{char p DML} (we also note the papers \cite{Derksen, Derksen-Masser}, which treat related problems).  In Section~\ref{subsection arithmetic sequences} we also discuss some properties of the $p$-sets, i.e., sets of the form \eqref{form of the solutions}. Then we proceed in Section~\ref{section reduction} to prove that Conjecture~\ref{char p DML} for regular self-maps of semiabelian varieties $X$ defined over a finite field reduces to a polynomial-exponential question (see Theorem~\ref{theorem reduction}). We stress out that our result applies to \emph{all} semiabelian varieties $X$; in particular, this would allow future researchers to use Theorem~\ref{theorem reduction} for solving additional cases of Conjecture~\ref{char p DML} for semiabelian varieties $X$ once there will be new progress to solving the polynomial-exponential equations corresponding to Conjecture~\ref{char p DML}.  In Section~\ref{section all polynomial-exponential equations}, we prove that in the case $X=\bG_m^N$, then we actually recover a \emph{large class} of possible polynomial-exponential equations (this is Theorem~\ref{theorem reduction to polynomial-exponential equations}); so, the results from Section~\ref{section all polynomial-exponential equations} may be viewed as complementary to the results from Section~\ref{section reduction}. In Section~\ref{section diophantine}, we establish the key Diophantine statement for our proofs of Theorems~\ref{surface result}~and~\ref{Frobenius result}.  We conclude our paper by proving Theorems~\ref{surface result}~and~\ref{Frobenius result} in Section~\ref{section proofs}.


\section{The Mordell-Lang problem in positive characteristic}
\label{section M-L p}


\subsection{Semiabelian varieties}
\label{subsection semiabelian}

From now on, $p$ is a prime number, $K$ is an algebraically closed field of characteristic $p$, and $X$ is a semiabelian variety defined over $K$, i.e., there exists a short exact sequence of algebraic groups defined over $K$: 
$$1\lra T\lra X\lra A\lra 1,$$
where $A$ is an abelian variety and $T$ is an algebraic torus (possibly trivial). Each regular self-map $\Phi$ of $X$ is a composition of a translation with an (algebraic group) endomorphism $\Phi_0$. Indeed, after composing $\Phi$ with a translation, we may assume $\Phi$ sends the identity of $X$ into itself. In particular, this means that the induced map $\varphi:=\Phi|_T$ is a regular self-map on the torus $T$, which maps the identity into itself; hence $\varphi$ is an endomorphism of $T$. Furthermore, the induced map $\bar{\Phi}:A\lra A$ also maps the identity into itself and so, it must be a group endomorphism as well (see \cite[Corollary~1.2]{Milne}). Since there is no non-constant regular map between an abelian variety and a torus, we derive that $\Phi$ must be a group endomorphism, as claimed.

Each endomorphism $\Phi_0:X\lra X$ is integral over $\Z$ (where $\Z$ is seen as a subring of the ring $\End(X)$ of endomorphisms of $X$), i.e., there exists $\ell\in\N$ and there exist integers $c_0,\dots, c_{\ell-1}$ such that
\begin{equation}
\label{general equation endomorphism}
\Phi_0^\ell + c_{\ell-1}\Phi_0^{\ell-1}+\cdots + c_1\Phi_0+c_0\cdot \Id = 0,
\end{equation}
where the above identity endomorphism $\Id$ is seen inside $\End(X)$. Actually, $\ell\le 2\dim(X)$, but this fact will not be necessary in our proof.

Assume $X$ is defined over a finite subfield $\Fq$ of $K$, and let $F$ be the corresponding Frobenius map for the finite field $\Fq$; we extend $F$ to an endomorphism of $X$. Then also $F$ satisfies an equation of the form \eqref{general equation endomorphism}; furthermore, as a consequence of Weil conjectures for abelian varieties over finite fields, we note that the roots of the corresponding minimal equation satisfied by the endomorphism $F$ are distinct and each one has absolute value equal to $q^{a_i}$ with $a_i\in\left\{\frac{1}{2},1\right\}$. Actually, we know more precise information about the rational numbers $a_i$ (i.e., there is at most one $a_i=1$ and if some $a_i=1$ then the corresponding root of \eqref{general equation endomorphism} equals $q$ and the torus $T$ is nontrivial, while there are $2\cdot \dim(A)$ roots of absolute value $q^{\frac{1}{2}}$, according to Weil conjectures \cite[Chapter~2]{Milne}); however, the only relevant information for us will be that the roots for the minimal equation satisfied by $F$ are distinct and their absolute values are of the form $q^{a_i}$ for some positive $a_i\in\Q$.   


\subsection{The intersection of a subvariety of a semiabelian variety defined over a finite field with a finitely generated subgroup}
\label{subsection F-sets}

We continue with the notation from Section~\ref{subsection semiabelian} for $K$ and $X$. We assume $X$ is defined over a finite subfield $\Fq\subset K$, and we let $V$ be an arbitrary subvariety of $X$ defined over $K$.  In order to state Theorem~\ref{Moosa-Scanlon theorem} (which is crucial for all our proofs), we need first to introduce the notion of \emph{$F$-sets} defined by Moosa-Scanlon \cite{Moosa-Scanlon}. The Frobenius $F$ acting on $X$ is the endomorphism induced by the usual field homomorphism given by $x\mapsto x^q$ for each $x\in K$.


\begin{definition} 
\label{definition F-sets}
With the above notation for $X$, $q$, $K$ and $F$, let $\Gamma\subseteq X(K)$ be a finitely generated subgroup.
\begin{itemize}
\item[(a)]
By a \emph{sum of $F$-orbits} in $\Gamma$ we mean a set of the form
$$C(\alpha_1,\dots,\alpha_m;k_1,\dots,k_m):=\left\{\sum_{j=1}^m F^{k_jn_j}(\alpha_j) \colon n_j\in\mathbb{N}_0\right\}\subseteq\Gamma$$
where $\alpha_1,\dots,\alpha_m$ are some given points in $X(K)$ and $k_1,\dots,k_m$ are some given non-negative integers.
\item[(b)]
An \emph{$F$-set} in $\Gamma$ is a set of the form 
$C+ \Gamma '$ where $C$ is a sum of $F$-orbits in $\Gamma$, and $\Gamma '\subseteq \Gamma$ is a subgroup, while in general, for two sets $A,B\subset X(K)$, then $A+ B$ is simply the set $\{a+ b\colon a\in A\text{,  }b\in B\}$. 
\end{itemize}
\end{definition}

As a matter of convention, when we work with $F$-orbits inside an algebraic torus $\bG_m^N$ (corresponding to the usual Frobenius map $x\mapsto x^p$), then we call \emph{product} of $F$-orbits, rather than sum of $F$-orbits; also, in this case we use the notation $\prod_{j=1}^m \alpha_j^{p^{k_jn_j}}$ for an element in the corresponding product of $F$-orbits (as in part~(a) of Definition~\ref{definition F-sets}).

We note that allowing for the possibility that some $k_i=0$ in the definition of a sum of $F$-orbits, implicitly we allow a singleton be a sum of $F$-orbits; this also explains why we do not need to consider cosets of subgroups $\Gamma '$ in the definition of an $F$-set. Furthermore, if $k_2=\cdots = k_m=0$ then the corresponding sum of $F$-orbits is simply a translate of the single $F$-orbit $C(\alpha_1;k_1)$. Note that the $F$-orbit $C(\alpha;k)$ for $k>0$ is finite if and only if $\alpha\in X\left(\overline{\mathbb{F}_p}\right)$. Finally, we call an orbit $C(\alpha;k)$ nontrivial, if it is infinite (i.e., $k>0$ and $\alpha\notin X\left(\overline{\mathbb{F}_p}\right)$).

We also note that since we allow the base points $a_i$ be outside $\Gamma$, we can use the slightly simpler definition of groupless $F$-sets involving sums of $F$-orbits rather than using the $F$-cycles (see \cite[Remark~2.6]{Moosa-Scanlon}, and also the slight extension proven in \cite{groups}).

\begin{theorem}[Moosa-Scanlon \cite{Moosa-Scanlon}]
\label{Moosa-Scanlon theorem}
Let $X$ be a semiabelian variety defined over $\Fq$, let $\Fq\subset K$ be an algebraically closed field, let $V\subset X$ be a subvariety defined over $K$ and let $\Gamma\subset X(K)$ be a finitely generated subgroup. Then $V(K)\cap \Gamma$ is a finite union of $F$-sets contained in $\Gamma$.
\end{theorem}

\begin{remark}
\label{remark s dim V}
One can deduce following the argument from \cite{Moosa-Scanlon} that each sum of $F$-orbits appearing in the interrsection of $V$ with the group $\Gamma$ is a sum of at most $\dim(V)$ nontrivial orbits; this follows easily from the proof of \cite[Theorem~7.8]{Moosa-Scanlon} using induction on the dimension of $V$ (see also \cite[Corollary~7.3]{Moosa-Scanlon}). Also, in the case $X$ is an algebraic torus, Derksen-Masser \cite[Theorem~1]{Derksen-Masser} (see also the paper by Masser \cite{Masser} on a similar topic) proved that there are at most $\dim(V)$ nontrivial orbits under the Frobenius in any product of $F$-orbits appearing in the intersection of $V\subset \bG_m^N$ with $\Gamma$. 

We state below a special case of these results which we will employ in our proof of Theorem~\ref{surface result}; the next corollary can be deduced also directly using a similar argument as in the proof of \cite[Corollary~2.3]{G-TAMS}.
\end{remark}

\begin{corollary}
\label{precise intersection surface}
With the notation from Theorem~\ref{Moosa-Scanlon theorem}, if $X=\bG_m^N$ and  $V$ is a surface, then the intersection $V(K)\cap\Gamma$ is a finite union of $F$-sets $C\cdot \Gamma '$, where $C$ is a product of  orbits under the Frobenius action on $\bG_m^N$ of which at most $2$ orbits are nontrivial, while $\Gamma '$ is a subgroup of $\Gamma$.
\end{corollary}


\subsection{Arithmetic sequences}
\label{subsection arithmetic sequences}

In this section (which overlaps with \cite[Section~3]{G-TAMS}) we state various useful results regarding linear recurrence sequences and $p$-sets.

A \emph{linear recurrence sequence} is a sequence $\{u_n\}_{n\in\N_0}\subset \C$ with the property that there exists $m\in\N$ and there exist $c_0,\dots, c_{m-1}\in\C$ such that for each $n\in\N_0$ we have
\begin{equation}
\label{equation definition linear recurrence sequence}
u_{n+m}+c_{m-1}u_{n+m-1}+\cdots + c_1u_{n+1}+c_0u_n=0.
\end{equation}

If $c_0=0$ in \eqref{equation definition linear recurrence sequence}, then we may replace $m$  by $m-k$ where $k$ is the smallest positive integer such that $c_k\ne 0$; then the sequence $\{u_n\}_{n\in\N_0}$ satisfies the linear recurrence relation
\begin{equation}
\label{equation definition linear recurrence sequence 2}
u_{n+m-k}+c_{m-1}u_{n+m-1-k}+\cdots + c_ku_{n}=0,
\end{equation}
for all $n\ge k$. In particular, if $k=m$, then the sequence $\{u_n\}_{n\in\N_0}$ is eventually constant, i.e., $u_n=0$ for all $n\ge m$. 

Assume now that $c_0\ne 0$ (which may be achieved at the expense of re-writing the recurrence relation as in \eqref{equation definition linear recurrence sequence 2}); then there exists a closed form for expressing $u_n$ for all $n$ (or at least for all $n$ sufficiently large if one needs to re-write the recurrence relation as in \eqref{equation definition linear recurrence sequence 2}). The characteristic roots of a linear recurrence sequence as in \eqref{equation definition linear recurrence sequence} are the solutions of the equation
\begin{equation}
\label{characteristic equation}
x^m+c_{m-1}x^{m-1}+\cdots +c_1x+c_0=0.
\end{equation}
We let $r_i$ (for $1\le i\le s$) be the (nonzero) roots of the equation \eqref{characteristic equation}; then there exist polynomials $P_i(x)\in\C[x]$ such that for all  $n\in\N_0$, we have
\begin{equation}
\label{general formula linear recurrence sequence}
u_n=\sum_{i=1}^s P_i(n)r_i^n.
\end{equation}
In general, as explained above, for an arbitrary linear recurrence sequence, the formula \eqref{general formula linear recurrence sequence} holds for all $n$ sufficiently large (more precisely, for all $n\ge m$ with the notation from \eqref{equation definition linear recurrence sequence}); for more details on linear recurrence sequences, we refer the reader to the chapter on linear recurrence sequences written by Schmidt \cite{Sch03}.

It will be convenient for us later on in our proof to consider linear recurrence sequences which are given by a formula such as the one in \eqref{general formula linear recurrence sequence} (for $n$ sufficiently large) for which the following two properties hold:
\begin{enumerate}
\item[(i)] if some $r_i$ is a root of unity, then $r_i=1$; and  
\item[(ii)] if $i\ne j$, then $r_i/r_j$ is not a root of unity.
\end{enumerate}   
Such linear recurrence sequences given by formula \eqref{general formula linear recurrence sequence} and satisfying properties~(i)-(ii) above are called \emph{non-degenerate}.   Given an arbitrary linear recurrence sequence, we can always split it into finitely many linear recurrence sequences which are all non-degenerate; moreover, we can achieve this by considering instead of one sequence $\{u_n\}$, finitely many sequences which are all of the form $\{u_{nM+\ell}\}$ for a given $M\in\N$ and for $\ell=0,\dots, M-1$. Indeed, assume some $r_i$ or some $r_i/r_j$ is a root of unity, say of order $M$; then for each $\ell=0,\dots, M-1$ we have that 
\begin{equation}
\label{re-writing u n}
u_{nM+\ell}=\sum_{i=1}^s P_i(nM+\ell)r_i^{\ell} (r_i^M)^n
\end{equation}
and moreover, we can re-write the formula \eqref{re-writing u n} for $u_{nM+\ell}$ by collecting the powers $r_i^M$ which are equal and thus achieve a non-degenerate linear recurrence sequence $v_n:=u_{nM+\ell}$. 

The following famous theorem of Skolem \cite{Skolem} (later generalized by Mahler \cite{Mahler} and Lech \cite{Lech})  will be used throughout our proof.
\begin{proposition}
\label{Skolem result}
Let $\{u_n\}_{n\in\N_0}\subset \C$ be a linear recurrence sequence, and let $c\in\C$. Then the set $T$ of all $n\in\N_0$ such that $u_n=c$ is a finite union of arithmetic progressions; moreover, if $\{u_n\}$ is a non-degenerate linear recurrence sequence, then the set $T$ is infinite only if the sequence $\{u_n\}$ is eventually constant. 
\end{proposition}

Next we consider the $p$-sets, which are sets of non-negative integers of the form \eqref{form of the solutions}, i.e., 
\begin{equation}
\label{p-sets again 2}
\left\{\sum_{j=1}^m c_j p^{k_jn_j}\colon n_j\in\N_0\right\} 
\end{equation}
for some $m\in\N$, some $c_j\in\Q$ and some $k_j\in\N_0$.  Since for each positive integer $M$, the powers of $p$ are preperiodic modulo $M$, we immediately obtain the following result.

\begin{proposition}
\label{prop:p-arithmetic and arithmetic}
Let $p$ be a prime number. 
The intersection of an arithmetic progression with a $p$-set is a finite union of $p$-sets. 
\end{proposition}

Also, as a simple application of Laurent's theorem \cite{Laurent} regarding the intersection of subvarieties of an algebraic torus $T$ with finitely generated subgroups of $T$, we obtain the following result.

\begin{proposition}
\label{prop:intersections of p-arithmetic sequences}
The intersection of two $p$-sets is a finite union of $p$-sets. More precisely, if each of the two $p$-sets that we intersect consists of sums of at most $m$ powers of $p$ (see \eqref{p-sets again 2}), then each $p$-set from their intersection also consists of sums of at most $m$ powers of $p$ as in \eqref{p-sets again 2}.
\end{proposition}


\section{Reduction of the dynamical problem to a polynomial-exponential equation}
\label{section reduction}

Our goal is to prove that solving Conjecture~\ref{char p DML} reduces to solving several polynomial-exponential equations (see Theorem~\ref{theorem reduction}). In order to state our result, we introduce some notation.

In this section, we fix some finite field $\Fq$ of characteristic $p$; also, we let $X$ be a semiabelian variety defined over $\Fq$. We let $F$ be the Frobenius endomorphism of $X$ induced by the field automorphism $x\mapsto x^q$. We let  $P_{{\rm min},F}\in\Z[x]$ be the minimal polynomial for the Frobenius (as an endomorphism of $X$); then $P_{{\rm min},F}$ has simple roots: $\lambda_1,\dots,  \lambda_r$. So, 
\begin{equation}
\label{min P F}
P_{{\rm min},F}\text{ is a monic polynomial of degree $r$ and $P_{{\rm min},F}(F)=0$}.
\end{equation}
Let $\{U_n\}_{n\in\N_0}$ be a linear recurrence sequence. Let $m\in\N_0$, let $\{U^{(i)}_{n}\}_{n\in\N_0}\subset \Z$ (for $i=1,\dots, m$) be linear recurrence sequences, each one of them having simple characteristic roots, all of them being of the form $\lambda_j^{b_{i}}$ for some $b_{i}\in\N_0$ (where the $\lambda_j$'s are the minimal roots of the minimal polynomial $P_{{\rm min},F}\in\Z[x]$ of the Frobenius endomorphism of $X$). In particular, if $b_i=0$ then $U^{(i)}$ is a constant sequence because it is a linear recurrence sequence which has only one characteristic root, which is equal to $1$. So, if $U^{(i)}$ is a constant sequence (i.e., if $b_i=0$), then we say that $U^{(i)}$ is a trivial sequence, while if $U^{(i)}$ is non-constant (i.e., if $b_i\in \N$), then we say that $U^{(i)}$ is a nontrivial sequence.

Then, with the above notation for the linear recurrence sequences $U$ and $U^{(i)}$, given also some $a,b\in\N_0$, we call an \emph{$F$-arithmetic sequence}, the set of all $n$ of the form $ak+b$ (for some $k\in\N_0$) for which there exist $n_1,\dots, n_m\in\N_0$ such that 
\begin{equation}
\label{F-arithmetic sequence}
U_n=U^{(1)}_{n_1}+\cdots +U^{(m)}_{n_m}.
\end{equation}

\begin{remarks}
\label{rem:F-arithmetic}
The following observations will be used repeatedly in the proof of Theorem~\ref{theorem reduction}.
\begin{itemize}
\item[(i)] If $m=0$, then Equation~\ref{F-arithmetic sequence} is void and in this case, the $F$-arithmetic sequence is simply the arithmetic progression $\{ak+b\}_{k\in\N_0}$.

\item[(ii)] On the other hand, if $a=0$ (regardless of $m\in\N_0$), then the corresponding $F$-arithmetic sequence contains at most one element $\{b\}$.

\item[(iii)] In general, we could combine the two conditions defining the $F$-arithmetic sequence into one condition:
\begin{equation}
\label{F-arithmetic sequence a k b}
U_{ak+b}=U^{(1)}_{n_1}+\cdots +U^{(m)}_{n_m}.
\end{equation}
However, we prefer to work with Equation~\ref{F-arithmetic sequence} instead and tacitly assume that the set of solutions $n$ must also belong to an arithmetic progression. Note that the sequence $\{U_{ak+b}\}_{k\in\N_0}$ is itself a linear recurrence sequence whose characteristic roots are the $a$-th powers of the characteristic roots of the linear recurrence sequence $U$.
\end{itemize}
\end{remarks}

The intersection of finitely many $F$-arithmetic sequences is called a \emph{general $F$-arithmetic sequence}. 

\begin{theorem}
\label{theorem reduction}
With the above notation for $p$, $q$, $X$ and $F$, let $\Phi:X\lra X$ be a regular self-map defined over an algebraically closed field $K$ containing $\Fq$, let $\alpha\in X(K)$ and let $V\subset X$ be a subvariety defined over $K$. Then the set of all $n\in \N_0$ such that $\Phi^n(\alpha)\in V(K)$ is a union of finitely many general $F$-arithmetic sequences.  

Furthermore, the following more precise statements hold:
\begin{enumerate}
\item[(1)] The number of nontrivial linear recurrence sequences $U^{(i)}$ appearing in the definition of each $F$-arithmetic sequence from the conclusion of Theorem~\ref{theorem reduction} (see \eqref{F-arithmetic sequence}) is at most equal to $\dim(V)$.
\item[(2)] If $\Phi$ is a group endomorphism, then the characteristic roots of the linear recurrence sequences $\{U_n\}$ appearing in the left-hand side of each equation \eqref{F-arithmetic sequence} defining any of the above general $F$-arithmetic sequences are positive integer powers of the roots of the minimal polynomial in $\Z[x]$ satisfied by $\Phi$ (see equation \eqref{general equation endomorphism}). 
\end{enumerate}
\end{theorem}

\begin{remark}
Since a (generalized) $F$-arithmetic sequence may be a singleton (see Remark~\ref{rem:F-arithmetic}~(ii)), our conclusion in Theorem~\ref{theorem reduction} allows for finite sets to appear as the return set of $n\in\N_0$ such that $\Phi^n(\alpha)\in V(K)$. Similarly, as noted in Remark~\ref{rem:F-arithmetic}~(i), we also allow the possibility of an arithmetic progression to appear as the return set in Theorem~\ref{theorem reduction}.
\end{remark}

\begin{remark}
\label{rem:more_precise}
Because we want to establish the more precise statement~(2) from the conclusion of Theorem~\ref{theorem reduction}, we prefer to use Equation~\ref{F-arithmetic sequence} rather than Equation~\ref{F-arithmetic sequence a k b} when working with an $F$-arithmetic sequence (see also Remark~\ref{rem:F-arithmetic}~(iii)).
\end{remark}

\begin{proof}[Proof of Theorem~\ref{theorem reduction}.]
Since $X$ is a semiabelian variety, then there exists a group endomorphism $\Phi_0:X\lra X$ and there exists $y\in X(K)$ such that for any $\gamma\in X$, we have that $\Phi(\gamma)=\Phi_0(\gamma)+y$. Also, as noted in \eqref{general equation endomorphism}, there exists $\ell\in\N$ such that $\Phi_0$ satisfies (in ${\rm End}(X)$) a monic equation of degree $\ell$ with integer coefficients. Then, as shown in \cite[Claim~4.2]{G-TAMS} (whose proof for algebraic tori extends verbatim for any semiabelian variety), there exist linear recurrence sequences $\{u^{(i)}_{n}\}_{n\in\N_0}\subset \Z$ for $1\le i\le \ell$ and $\{v_{n}^{(i)}\}_{n\in\N_0}\subset \Z$ for $0\le i\le \ell-1$ such that for each $n\in\N_0$, we have
\begin{equation}
\label{iterate n Phi}
\Phi^n(\alpha)=\sum_{i=1}^\ell u^{(i)}_{n}\left(\sum_{j=0}^{i-1} \Phi_0^j(y)\right) + \sum_{i=0}^{\ell-1} v^{(i)}_{n}\Phi_0^i(\alpha).
\end{equation}
Furthermore, the characteristic roots of the linear recurrence sequences $v^{(i)}$ are among the roots of the minimal polynomial $P_{{\rm min},\Phi_0}\in\Z[x]$ of $\Phi_0$, while the characteristic roots of $u^{(i)}$ are contained in the set consisting of $1$ and all of the roots of $P_{{\rm min},\Phi_0}$ (see  \cite[Equations~(14)~and~(17)]{G-TAMS}). 

We also use the following notation
\begin{equation}
\label{iterate n Phi 2}
Q_i:=\sum_{j=0}^{i-1}\Phi_0^j(y)\text{ for each $i=1,\dots, \ell$.}
\end{equation}

We let $\Gamma$ be a finitely generated group containing $\Phi_0^j(y)$ and $\Phi_0^j(\alpha)$ for each $j=0,\dots, \ell-1$. Theorem~\ref{Moosa-Scanlon theorem} yields that 
\begin{equation}
\label{first description intersection}
V(K)\cap\Gamma\text{ is a finite union of $F$-sets $C+H$ contained in $\Gamma$,}
\end{equation}
where $C$ is a sum of $F$-orbits (actually, we have at most $\dim(V)$ nontrivial orbits appearing in $C$, as observed in Remark~\ref{remark s dim V}---see also Corollary~\ref{precise intersection surface}) and $H$ is a subgroup of $\Gamma$. Thus we write 
\begin{equation}
\label{sum F orbit base points}
C=\left\{\sum_{j=1}^{m} F^{n_jk_j}(\gamma_j)\colon n_j\in\N_0\right\},
\end{equation}
for some $m\in\N_0$, some $\gamma_j\in X(K)$ and some $k_j\in\N_0$, and moreover, there are at most $\dim(V)$ nonzero $k_i$'s.  

At the expense of replacing $\Gamma$ by a larger (but still finitely generated) group, we may assume $\Gamma$ contains each $F^i(\gamma_j)$ for $j=1,\dots, m$ and $i=0,\dots, r-1$. In particular, we get that $F^i(\gamma_j)\in \Gamma$ for all $i\ge 0$ (see \eqref{min P F}).

Since $\Gamma$ is a finitely generated abelian group, we know that it is isomorphic to a direct sum of a finite subgroup $\Gamma_0$ with a subgroup $\Gamma_1$ which is isomorphic to $\Z^s$ for some $s\in\N_0$. We let $\{P_1,\dots, P_s\}$ be a $\Z$-basis for $\Gamma_1$. 

We proceed with the notation from equations \eqref{iterate n Phi} and  \eqref{iterate n Phi 2} and so, 
$$\Phi^n(\alpha)=\sum_{i=1}^\ell u^{(i)}_{n}Q_i + \sum_{i=0}^{\ell-1}v^{(i)}_{n}\Phi_0^i(\alpha).$$  
Then we write each $Q_i$ (for $i=1,\dots, \ell$) as $Q_{i,0}+\sum_{j=1}^s b_{i,j}P_j$ with $Q_{i,0}\in \Gamma_0$ and each $b_{i,j}\in\Z$, and also write each $\Phi_0^i(\alpha)$ (for $i=0,\dots, \ell-1$) as $T_{i,0}+\sum_{j=1}^s c_{i,j}P_j$ with $T_{i,0}\in \Gamma_0$ and each $c_{i,j}\in\Z$.  

We also write for each $j=1,\dots, m$  and for each $i=0,\dots, r-1$ (note that $r$ is the degree of the minimal polynomial $P_{{\rm min},F}$ with integer coefficients satisfied by the Frobenius endomorphism in $\End(X)$; see equation \eqref{min P F})  
$$F^i(\gamma_j)=R^{(i)}_{j,0}+\sum_{k=1}^s d_{i,j,k}P_k\text{ with }R^{(i)}_{j,0}\in \Gamma_0\text{ and each }d_{i,j,k}\in\Z.$$
We will write the conditions satisfied by $n\in\N_0$ in order for $\Phi^n(\alpha)\in (C + H)$ (see \eqref{first description intersection}), or equivalently that there exist some $n_1,\dots, n_m\in\N_0$ such that 
$$\Phi^n(\alpha) - \sum_{j=1}^m F^{k_jn_j}(\gamma_j) \in H.$$  
In order to do this, we observe that, similar to deducing the formula \eqref{iterate n Phi}, we have that there exist linear recurrence sequences $\{a^{(i)}_n\}_{n\in\N_0}\subset \Z$ for $i=0,\dots, r-1$ such that
\begin{equation}
\label{iterate n F}
F^n(\gamma_j)=\sum_{i=0}^{r-1} a^{(i)}_n F^i(\gamma_j).
\end{equation}
Furthermore, for each $i=0,\dots, r-1$, the characteristic equation satisfied by the linear recurrence sequence $\{a^{(i)}_n\}_{n\in\N_0}$ is $P_{{\rm min},F}(x)=0$. Indeed, this follows from the fact that $P_{{\rm min},F}(F)=0$ and then expressing reccursively $F^n$ as linear combinations of $F^i$ (for $i=0,\dots, r-1$) with integer coefficients $a^{(i)}_n$ and then observing that each sequence $\{a^{(i)}_n\}_{n\in\N_0}$ satisfies the linear recurrence given by the characteristic equation $P_{{\rm min},F}(x)=0$. 

We compute then
\begin{eqnarray*}
& & \Phi^n(\alpha)-\sum_{i=1}^m F^{k_in_i}(\gamma_i)\\
& = & \left( \sum_{i=1}^\ell u^{(i)}_{n}Q_{i,0} + \sum_{i=0}^{\ell-1} v^{(i)}_{n}T_{i,0} - \sum_{j=0}^{r-1}\sum_{i=1}^m a^{(j)}_{k_in_i} R_{i,0}^{(j)} \right)\\
& + & \sum_{k=1}^s \left( \sum_{i=1}^\ell b_{i,k}u^{(i)}_{n} + \sum_{i=0}^{\ell-1} c_{i,k} v^{(i)}_{n} - \sum_{j=0}^{r-1} \sum_{i=1}^m d_{j,i,k}a^{(j)}_{k_in_i}\right) \cdot P_k.
\end{eqnarray*}
Since linear combinations of linear recurrence sequences are again linear recurrence sequences (whose characteristic roots are among the characteristic roots of the original linear recurrence sequences), we conclude that there exist linear recurrence sequences $\{U^{(k)}_n\}_{n\in\N_0}$ for $k=1,\dots, s$ and $\{A^{(i,k)}_n\}_{n\in\N_0}$ for $i=1,\dots, m$ and $k=1,\dots, s$ such that
\begin{eqnarray*}
& & \Phi^n(\alpha)-\sum_{i=1}^m F^{k_in_i}(\gamma_j)\\
& = & \left( \sum_{i=1}^\ell u^{(i)}_{n}Q_{i,0} + \sum_{i=0}^{\ell-1} v^{(i)}_{n}T_{i,0} - \sum_{j=0}^{r-1}\sum_{i=1}^m a^{(j)}_{k_in_i} R_{i,0}^{(j)} \right)\\
& + & \sum_{k=1}^s \left(U^{(k)}_n - \sum_{i=1}^m A^{(i,k)}_{n_i}\right) P_k.
\end{eqnarray*}
More precisely, $U^{(k)}_n:= \sum_{i=1}^\ell b_{i,k}u^{(i)}_{n} + \sum_{i=0}^{\ell-1} c_{i,k} v^{(i)}_{n}$. In particular, if $\Phi=\Phi_0$ (i.e., if $y=0$) then $U^{(k)}_n=\sum_{i=0}^{\ell-1} c_{i,k}v^{(i)}_n$ and so, the characteristic roots of each $U^{(k)}$ is among the characteristic roots of the $v^{(i)}$'s, i.e., the characteristic roots of the $U^{(k)}$'s are among the roots of $P_{{\rm min},\Phi}$. Also, for each $i=1,\dots, m$, each $k=1,\dots, s$ and for each $n\in\N_0$, we have that $$A^{(i,k)}_n:=\sum_{j=0}^{r-1}d_{j,i,k}a^{(j)}_{k_in}.$$  
In particular, if $k_i=0$ (i.e., the corresponding $F$-orbit $C(\gamma_i;k_i)$ is trivial), then $A^{(i,k)}$ is a trivial linear recurrence sequence (i.e., it is constant). Furthermore, the  characteristic roots for each linear recurrence sequence $A^{(i,k)}$ are $k_i$-th powers of the eigenvalues $\lambda_j$ of the Frobenius $F$. 

Now, since each $Q_{i,0}$, $T_{i,0}$ and $R^{(j)}_{i,0}$ is a torsion point (of order bounded by the cardinality $M$ of the torsion subgroup of $\Gamma$) and moreover, any linear recurrence sequence of integers is preperiodic modulo $M$, then at the expense of replacing each $n$ and also each $n_i$ by suitable arithmetic progressions (which, in particular, leads to replacing $U^{(k)}$ and $A^{(i,k)}$ by other linear recurrence sequences whose characteristic roots are powers of the characteristic roots for the original linear recurrence sequences), we may assume there exists some torsion point $S_0\in\Gamma$ such that
\begin{eqnarray*}
& & \Phi^n(\alpha)-\sum_{j=1}^m F^{k_jn_j}(\gamma_j)\\
& = & S_0 + \sum_{k=1}^s \left(U^{(k)}_n - \sum_{i=1}^m A^{(i,k)}_{n_i}\right) P_k.
\end{eqnarray*}
Replacing $n$ by an arithmetic progression $\{an+b\}_{n\in\N_0}$ (for some $a\in\N$ and $b\in\N_0$) simply changes $\Phi$ by $\Phi^a$ and also replaces $\alpha$ by $\Phi^b(\alpha)$. Also, the conclusion of Theorem~\ref{theorem reduction} is unaffected by this change since for any linear recurrence sequence $\{w_n\}_{n\in\N_0}$, the sequence $\{w_{an+b}\}_{n\in\N_0}$ satisfies also a linear recurrence and its characteristic roots are the $a$-th powers of the characteristic roots of $\{w_n\}_{n\in\N_0}$; see also  Remark~\ref{rem:F-arithmetic}~(iii). 

So, we need to determine for which $n\in\N_0$ there exist some $n_1,\dots, n_m\in\N_0$ such that $\Phi^n(\alpha)-\sum_{j=1}^m F^{k_jn_j}(\gamma_j)\in H$, or equivalently, for which $n\in\N_0$ we have that
\begin{equation}
\label{condition to be in H}
\sum_{k=1}^s \left(U^{(k)}_n - \sum_{i=1}^m A^{(i,k)}_{n_i}\right) P_k \in -S_0+H.
\end{equation}
Now, according to \cite[Claim 3.4, Definition 3.5, Subclaim 3.6]{groups}, the condition \eqref{condition to be in H} is equivalent with a system formed by finitely many equations either of the form
\begin{equation}
\label{condition to be in H equality}
\sum_{k=1}^s e_k\cdot \left(U^{(k)}_n - \sum_{i=1}^m A^{(i,k)}_{n_i}\right)=0,
\end{equation}
for some integers $e_k$, or of the form
\begin{equation}
\label{condition to be in H congruence}
\sum_{k=1}^s f_k\cdot \left(U^{(k)}_n - \sum_{i=1}^m A^{(i,k)}_{n_i}\right)\equiv 0\pmod{L}
\end{equation}
for some integers $f_k$ and some $L\in \N$. The solutions $(n,n_1,\cdots, n_m)$ to each congruence equation \eqref{condition to be in H congruence} is a finite union of sets of the form 
$$\{(an+b,a_1n_1+b,\cdots, a_mn+b_m)\colon n,n_1,\cdots,n_m\in\N_0\}$$
for some $a,a_i,b,b_i\in\N_0$ (for $i=1,\dots, m$). Therefore, again refining our analysis to finitely many arithmetic progressions, we may reduce to the case that $n$ (along with $n_1,\dots, n_m$) satisfies finitely many equations of the form \eqref{condition to be in H equality}. 


Each equation of the form \eqref{condition to be in H equality} reduces to 
\begin{equation}
\label{condition to be in H 2}
W_n=\sum_{i=1}^m B^{(i)}_{n_i},
\end{equation} 
where $W_n:=\sum_{k=1}^s e_kU^{(k)}_n$ and $B^{(i)}_n:=\sum_{k=1}^s e_k\cdot A^{(i,k)}_n$. So, $\{W_n\}_{n\in\N_0}$ is another linear recurrence sequence, while $\{B^{(i)}_n\}_{n\in\N_0}$ are linear recurrence sequences whose characteristic roots are the same as the characteristic roots of the original sequences $A^{(i,k)}$. In particular, the characteristic roots of $B^{(i)}$ are of the form $\lambda_j^{b_i}$ for some $b_i\in\N_0$, where the $\lambda_j$'s are the roots of $P_{{\rm min},F}$; note that $b_i=0$ if some $k_i=0$ and so, there are at most $\dim(V)$ nontrivial linear recurrence sequences $B^{(i)}$ as we vary $i=1,\dots, m$.  Furthermore, the characteristic roots of the linear recurrence sequence $\{W_n\}$ are positive integer powers of the roots of $P_{{\rm min},\Phi_0}$ and perhaps, also equal to $1$ if $\Phi$ is not a group endomorphism (i.e., $y\ne 0$); on the other hand, if $\Phi=\Phi_0$ (i.e., $y=0$), then the characteristic roots of $W$ are all positive integer powers of the roots of $P_{{\rm min},\Phi}=P_{{\rm min},\Phi_0}$. Finally, we note that if $m=0$ then Equation~\ref{condition to be in H 2} is void and therefore, the corresponding $F$-arithmetic sequence is simply an arithmetic progression; see Remark~\ref{rem:F-arithmetic}~(i).

This concludes our proof of Theorem~\ref{theorem reduction}.
\end{proof}


\section{Recovering a large class of polynomial-exponential equations as part of proving Conjecture~\ref{char p DML}}
\label{section all polynomial-exponential equations}

In this section we use a coding trick from~\cite{SY2014} to show that 
not only does the dynamical Mordell-Lang problem reduce to 
solving polynomial-exponential equations, but conversely complicated 
polynomial-exponential equations may be realized as instances of the
dynamical Mordell-Lang problem.

We start by extracting a key fact from~\cite{SY2014}.

\begin{proposition}
\label{encode lrs}
Let $\{ u_n \}$ be a linear recurrence sequence of integers, $G$ a 
commutative algebraic group over some field $K$ and $P \in G(K)$ 
a $K$-rational point on $G$.  Then there are a natural number
$N \in \N$, morphisms of algebraic groups $\Phi:G^N \to G^N$ and 
$\pi:G^N \to G$, and a point $Q \in G^N(K)$ so that for every $n \in \N_0$
we have $\pi ( \Phi^n (Q)) = [u_n]_G (P)$.	
\end{proposition}

\begin{proof}
By Proposition 4.3 and Theorem 4.5 of~\cite{SY2014}	there are a 
natural number $N \in \N$, a vector $a \in \Z^N$, and linear maps 
$\phi:\Z^N \to \Z^N$ and $\varpi:\Z^N \to \Z$ so that for every natural 
number $n \in \N_0$ we have $\varpi(\phi^n (a)) = u_n$.   Write 
$a$ as $a = (a_1, \ldots, a_N)$ and $\phi$ and $\varpi$ as matrices, 
so that $\phi(x_1, \ldots,x_N) = (\sum_{j=1}^N \phi_{1,j} x_j,
\ldots, \sum_{j=1}^N \phi_{N,j} x_j)$ and $\varpi(x_1, \ldots, x_N) 
= \sum_{j=1}^N \varphi_j x_j$.  Set $Q := ([a_1]_G(P), 
\ldots, [a_N]_G(P))$, where $a:=(a_1,\dots, a_N)$,  let $\Phi:G^N \to G^N$ be the map 
$(x_1, \ldots, x_N) \mapsto ( \sum_{j=1}^N 
[\phi_{1,j}]_G(x_j),  \ldots, \sum_{j=1}^N [\phi_{N,j}]_G(x_j) )$
and let $\pi:G^N \to G$ be the map $(x_1, \ldots, x_N) \mapsto 
\sum_{j=1}^N [\varpi_j]_G(x_j)$.   Then for $n \in \N$ we have 
$\pi (\Phi^n (Q)) = [u_n]_G(P)$. 
\end{proof}

It follows from Proposition~\ref{encode lrs} that we 
can realize many sets defined by conditions of the 
form $u_n \in E$ where $\{ u_n \}$ is a 
linear recurrence sequence of integers and $E \subseteq \N$ is a $p$-set 
as dynamical return sets.  More precisely, we have the 
following proposition.

\begin{proposition}
\label{prop:lrs plus ML}
Let $\{ u_n \}$ be a linear recurrence sequence of integers, $G$  
a commutative algebraic group over the field $K$, $X \subseteq G$ 
a subvariety, $P \in G(K)$ a $K$-rational point on $G$, and 
$E := \{ n \in \Z ~:~ [n]_G(P) \in X(K)\}$.  Then there is an algebraic
dynamical system $(Y,\Phi)$ over $K$, a point $Q \in Y(K)$ and 
a subvariety $Z \subseteq Y$ so that for all $n \in \N_0$ one has 
$\Phi^n(Q) \in Z(K) \Longleftrightarrow u_n \in E$.
\end{proposition}
\begin{proof}
Let $N$, $\Phi$, $\pi$ and $Q$ be given by Proposition~\ref{encode lrs}
applied to $\{ u_n \}$, $G$, and $P$.  Set $Y := G^N$ and $Z := \pi^{-1}(X)$; this yields the conclusion of Proposition~\ref{prop:lrs plus ML}.
\end{proof}

Many $p$-sets are known to be realizable as the exponent sets 
for instances of the Mordell-Lang problem in characteristic 
$p$.  With the next proposition, which is 
inspired by special examples constructed in~\cite{Ne17},
 we note that it is true in 
particular for sets of the form $\{ \sum_{i=1}^\ell c_i p^{n_i} : 
(n_1, \ldots, n_\ell) \in \N_0^\ell \}$ with $c_1, \ldots, c_\ell \in \N$ 
provided that 
$\sum_{i=1}^\ell c_i < p - 1$.

\begin{proposition}
\label{prop:p-set as ML}
Let $\ell \in \N$, $c_1, \ldots, c_\ell \in \N$ and $p \in \N$ be a
prime number satisfying $\sum_{i=1}^\ell c_i < p -1$.  Then there is an
algebraic torus $G$ over $\F_p$, a subvariety $X \subseteq G$ over 
$\F_p$ and a point $P \in G(\F_p(t))$ so that if 
$E := \{ m \in \Z ~:~ [m]_G(P) \in X(\F_p(t)) \}$, then 
$E = \{ \sum_{j=1}^\ell c_j p^{n_j} ~:~ (n_1,\ldots,n_m) \in \N_0^\ell \}$. 	
\end{proposition}

\begin{proof}
We begin with the case of $c_1 = c_2 = \cdots = c_\ell = 1$ and 
then use this result to deduce the general result.  

The Vandermonde matrix 
$$\big( a^j \big)_{\substack{ 1 \leq a \leq p-1 \\
0 \leq j < p-1}} \text{ ,}$$ regarded as a $(p-1) \times (p-1)$
matrix over $\F_p$, is invertible.  Write its inverse as  
$$ \big( A_{j,a} \big)_{\substack{1 \leq a \leq p-1 \\
0 \leq j < p-1}} $$  
so that for $0 \leq k < p-1$ and
$j \in \Z$ we have


\begin{equation}
\label{vdm eq}	
\sum_{a \in \F_p^\times} A_{k,a} a^j  =  
\left\{\begin{array}{cc}
 1 & \text{ if } j \equiv k \pmod{p-1} \\
 0 & \text{ otherwise.}	
 \end{array}\right.
\end{equation}

Let $G := \G_m^{p-1}$.  Set $P := (t+1, t+2, 
\ldots, t+p-1) \in G(\F_p(t))$.  Let $X \subseteq G$ be the subvariety 
defined by the equations  
$$\sum_{i=1}^{p-1} A_{\ell,i} x_i = 1\text{ and }\sum_{i=1}^{p-1} A_{k,i} x_i = 0\text{ for }\ell < k \le p-1.$$  
It is easy to 
see that $X$ is the variety parametrized by the map  
\begin{equation}
\label{map pi}
\pi:(y_1, \ldots, y_\ell) \mapsto \left( \prod_{j=1}^\ell (y_j + 1), 
\prod_{j=1}^\ell (y_j + 2), \ldots, \prod_{j=1}^\ell (y_j + p-1) \right).
\end{equation}     
In particular, if $n_1, \ldots, n_\ell \in \N_0$ and 
$m = \sum_{i=1}^\ell p^{n_i}$,  then 
\begin{align*}
[m]_G(P) & =  ( (t + 1)^m, \ldots, (t + p - 1)^m)\\
& = 
\left( (t +1)^{\sum_{i=1}^\ell p^{n_i}}, \ldots, (t + p - 1)^{\sum_{i=1}^\ell p^{n_i}}\right)\\
& = 
\left( \prod_{i=1}^\ell (t+1)^{p^{n_i}}, \ldots, \prod_{i=1}^\ell (t + p -1)^{p^{n_i}}\right) \\
& = 
\left( \prod_{i=1}^\ell (t^{p^{n_i}} + 1), \ldots, \prod_{i=1}^\ell (t^{p^{n_i}} + p - 1) \right) 
\in X(\F_p(t)).
\end{align*}   

Let us now show that if $[m](P) \in X(\F_p(t))$, then 
$m$ is a sum of $\ell$ powers of $p$. 
Considering poles of $(t+i)^m$ (for $i=1,\dots, p-1$), one sees that if $m < 0$, then $[m](P) \notin X(\F_p(t))$.  
Indeed, taking $b \in \F_p^\times$ for which $A_{\ell,b} \neq 0$, we see
that if $m < 0$, then the order of the pole at $-b$ of $\sum_{a \in \F_p^\times} 
A_{\ell,a} (t+a)^m$ is $-m$; so,  in particular, this sum is not equal to $1$.  

Take now $m \geq 0$.   Expanding, we have 

\begin{eqnarray*}
1 & = & \sum_{a \in \F_p^\times} A_{\ell,a} (t+a)^m \\
	& = & \sum_{a \in \F_p^\times} A_{\ell,a} \sum_{j=0}^m \binom{m}{j} a^j t^{m-j} \\
	& = & \sum_{j=0}^m t^{m-j} \binom{m}{j} \sum_{a \in \F_p^\times} A_{\ell,a} a^j 	
\end{eqnarray*}

Equating the coefficients of the powers of $t$, we see that 
\begin{equation}
\label{sum to one}
1 = \sum_{a \in \F_p^\times} A_{\ell,a} a^m
\end{equation}
and that for $j < m$ we have 
\begin{equation}
\label{sum to zero}
0 = \binom{m}{j} \sum_{a \in \F_p^\times} A_{\ell,a} a^j \text{ .}
\end{equation}
In particular, 
Equation~\ref{sum to one} implies that 
\begin{equation}
\label{m congruent l p-1}
m \equiv \ell \pmod{p - 1}.
\end{equation} 
Let us write $m$ in its base $p$ expansion, so that 
$$m = \sum_i m(i,p) p^i\text{ with }0 \leq m(i,p) < p.$$
Then $m\equiv \sum_i m(i,p)\pmod{p-1}$ and therefore, congruence equation \eqref{m congruent l p-1} yields that 
\begin{equation}
\label{m at least l}
\sum_i m(i,p)\ge \ell.
\end{equation}
Therefore, in order to obtain the desired conclusion in this special case of Proposition~\ref{prop:p-set as ML} (when $c_1=\cdots=c_\ell=1$), it suffices to prove that $\sum_i m(i,p)\le \ell$. 

Now, assuming we have a strict inequality in \eqref{m at least l}, then we can find some positive integer $j$ having a 
base $p$ expansion strictly dominated by that of $m$, that is,  if we write 
$$j = \sum_i j(i,p) p^i\text{ with }0 \leq j(i,p) < p,$$ 
then for all $i$ we have $j(i,p) \leq m(i,p)$ but 
$j(i,p) < m(i,p)$ for some $i$, and moreover, we can choose $j$  so that $j \equiv \ell \pmod{p -1}$.     


As Gau{\ss} observed (see~\cite[Lemma 15.21]{Eisenbud}), $\binom{m}{j} 
\not \equiv 0 \pmod{p}$.  
Thus, from 
Equation~\ref{sum to zero}, we see that $\sum_{a \in \F_p^\times} A_{\ell,a} a^j = 0$, but 
as $j \equiv \ell$, this contradicts Equation~\ref{vdm eq}.  Hence, 
$\sum m(i,p) \le \ell$, which combined with inequality \eqref{m at least l}, yields that $m$ must be the sum of exactly $\ell$ powers of $p$.

For the general case (i.e., arbitrary $c_i$ satisfying the inequality $\sum_{i=1}^\ell c_i<p-1$), we first adjust the notation somewhat taking $\ell' := 
\sum_i c_i$ and let $X' \subseteq \Gm^{p-1}$ be the variety constructed above corresponding to $\ell ' < p-1$.   
Let $X \subseteq X'$ be the subvariety parametrized by the map 
$$\nu:(y_1, \ldots, y_\ell) \mapsto \left( \prod_{j=1}^\ell (y_j + 1)^{c_j}, 
\prod_{j=1}^\ell (y_j + 2)^{c_j}, \ldots, \prod_{j=1}^\ell (y_j + p - 1)^{c_j}\right)$$ 
which may be seen as the composite of $\pi$ with a multidiagonal map from 
$\mathbb{A}^\ell$ to $\mathbb{A}^{\ell'}$.    As before, we take 
$P = ((t+1), (t+2), \ldots, (t+p-1))$.   Visibly, if $m$ is expressible
in the form $\sum_i c_i p^{n^i}$, then $[m](P) \in X(\F_p(t))$.  Conversely, 
for a general integer $m$, if $[m](P) \in X(\F_p(t))$, then because 
\emph{a fortiori} $[m](P) \in X'(\F_p(t))$ we may express $m$ as a sum of 
$\ell'$ powers of $p$.  Then the preimage of $[m](P)$ under the map $\pi$ from \eqref{map pi} consists of
sequences of the form $\left(t^{p^{n_1}}, \ldots, t^{p^{n_{\ell'}}}\right)$.  Since
$[m](P)$ is also in the image of $\nu$, the exponents must repeat $c_1$ times, and 
then $c_2$ times and so on, so that $m$ may be expressed in the 
form $\sum_{j=1}^\ell c_j p^{n_j}$.
\end{proof}
  
Combining Propositions~\ref{prop:lrs plus ML} and~\ref{prop:p-set as ML}  we deduce  Theorem~\ref{theorem reduction to polynomial-exponential equations} 
announced in the introduction.


\section{A diophantine result}
\label{section diophantine}
 
The following result is key for our proofs of Theorems~\ref{surface result} and \ref{Frobenius result}.
\begin{theorem}
\label{theorem polynomial-exponential}
Let $\{u_n\}_{n\in\N_0}$ be a linear recurrence sequence, let $p$ be a prime number, let $m\in\N$, let $c_1,\dots, c_m\in\Z$ and let $k_1,\dots, k_m\in\N_0$.
\begin{enumerate}
\item[(A)] If $p$ is multiplicatively independent with respect to each one of the characteristic roots of the linear recurrence sequence $\{u_n\}$ (i.e., for each characteristic root $\lambda$ of $\{u_n\}$ and for any integers $r$ and $s$, if $\lambda^r=p^s$, then $r=s=0$), then the set of all $n\in\N_0$ for which there exist some $n_1,\dots, n_m\in\N_0$ such that 
\begin{equation}
\label{u in T}
u_n=c_1p^{k_1n_1}+\cdots + c_mp^{k_mn_m}
\end{equation}
is a finite union of arithmetic progressions. (As always when dealing with an arithmetic progression $\{an+b\}_{n\in\N_0}$, we allow the possibility that $a=0$, in which case, the arithmetic progression becomes a singleton.)
\item[(B)] If $m\le 2$, then the set of all $n\in\N_0$ for which there exist some $n_1,n_2\in\N_0$ such that $u_n=c_1p^{k_1n_1}+c_2p^{k_2n_2}$ is a union of finitely many arithmetic progressions along with finitely many sets of the form
\begin{equation}
\label{set 2 p}
\left\{d_0+d_1p^{\ell_1n_1}+d_2p^{\ell_2n_2}\colon n_1,n_2\in\N_0\right\},
\end{equation}
for some rational numbers $d_0,d_1,d_2$, and some non-negative integers $\ell_1,\ell_2$. (In particular, if $\ell_1=\ell_2=0$, the set from \eqref{set 2 p} is a singleton.)
\end{enumerate}
\end{theorem}

\begin{remark} 
The proof presented below yields much more than required for the present purposes; indeed, we can conclude finiteness of  the set of solutions except in a few well-described cases.

In particular, the proof for Part (B)  implicitly gives that only the cases  when $u_n$ is essentially a polynomial with at most two distinct roots may produce infinitely many solutions, provided $k_i\neq 0$ for $i=1,2$ in Theorem~\ref{theorem polynomial-exponential}~(B) (as we may clearly assume). 

We have not pursued in obtaining and stating optimal conclusions, because in the first place this is not needed by the present application and anyway the interested reader shall readily see what can be extracted from the arguments. 
\end{remark}

We split our proof of Theorem~\ref{theorem polynomial-exponential} into the two parts of its conclusion.

\begin{proof}[Proof of Theorem~\ref{theorem polynomial-exponential}, part (A)]
We prove part (A), showing how this follows directly from a theorem of M. Laurent, whose statement we borrow from W.M. Schmidt's presentation \cite{Sch03} (see Thm. 7.1 therein).

We write $u_n$ as an exponential polynomial (as in Section~\ref{subsection arithmetic sequences}):
\begin{equation}
u_n=\sum_{i=1}^sP_i(n)r_i^n,
\end{equation}
where the $P_i$ are nonzero polynomials and the $r_i\in \C^*$ are pairwise distinct.  For our purpose we may and shall  assume that all the involved numbers are algebraic numbers.\footnote{The case when this is not verified may be actually treated in a simpler way.}

In proving the theorem we may partition $\N_0$ into arithmetical progressions of any prescribed modulus and hence we may assume without loss of generality that the recurrence is non-degenerate (i.e. no ratio $r_i/r_j$ is a root of unity different from $1$).

Under this assumption, we prove that in fact we have only finitely many solutions.

We write our equation as
\begin{equation}\label{E.sum}
\sum_{i=1}^sP_i(n)r_i^n-\sum_{j=1}^mc_jp^{k_jn_j}=0,
\end{equation}
where $m,c_j,k_j$ are fixed. So our variables are $n,n_1,\ldots ,n_m$, which are supposed to vary in $\N_0$. 
For each solution we may partition the $s+m$ terms on the left side into subsets such that the sum over each subset vanishes; we may also suppose that each subset is minimal with respect to this property. 

Since the number of possible partitions is finite, we may suppose on reducing $s,m$ (but still assuming $s>0$) that  no proper subsum of the left side of \eqref{E.sum} vanishes. 

In this situation, Laurent's theorem takes into account the subgroup of $\Z^{m+1}$ consisting of the integer vectors $(a_0,a_1,\ldots ,a_m)$ for which, 
\begin{equation}
\label{eq:mult_dep}
r_\mu^{a_0}=r_\nu^{a_0}, \qquad r_\mu^{a_0}=p^{k_ja_j},
\end{equation}
for each $\mu,\nu\in\{1,\ldots ,s\}$ and each $j=1,\ldots ,m$.  The equations on the left of \eqref{eq:mult_dep} may give no information, since $s$ could be possibly $1$, but the assumption on multiplicative independence of $p$ and any of the roots $r_i$, applied to the set of equations on the right of \eqref{eq:mult_dep},  yields that $a_0=a_1=\ldots =a_m=0$, i.e. our group reduces to $0$. 

Then Laurent's theorem asserts that the set of solutions is indeed finite, as required.
\end{proof}

\begin{remark} 
\label{rem:last_remark}
It is to be observed that, given that $u_n$ is non-degenerate,  in order to obtain finiteness we used merely the existence of {\it one} of the roots $r_i$ multiplicatively independent of $p$. It may thus seem that the stronger assumption of the independence for {\it all} roots is not needed. The point is however that  in going to the non-degenerate case the assumption could be lost if it only held for a subset of the roots.
\end{remark}

\begin{proof}[Proof of Theorem~\ref{theorem polynomial-exponential}, part (B)]
This case is treated by using implicitly the Subspace Theorem (as is the case for Laurent's theorem). The technique has been applied in several papers of the first and fourth authors; see e.g. \cite{CZ1}. This method works here as well; however, unfortunately we cannot directly appeal to  the results proved before, because of slight differences in the assumptions and context. Therefore we develop again the same technique for the case in question. In short, this consists of expanding the relevant solutions in convergent power series in several variables, and then suitably applying the Subspace Theorem to the approximations so derived, using each time a suitable absolute value. At least we shall not need to repeat the use of the Subspace Theorem, and shall use instead a theorem proved in \cite{CZ}.  The need  that the variables $x_i$ appearing in \cite{CZ} are such that  the ratios $\log |x_i|/\log |x_j|$ be bounded above and below by fixed positive numbers (for the relevant absolute values) forces us to separate the proof into two cases, using an archimedean and a $p$-adic absolute value respectively.

As before we can argue on separate arithmetical progressions and thus we can assume at the outset that $u_n$ is non degenerate.

Arguing as before (see also Remark~\ref{rem:last_remark})  we can assume that each of the roots $r_i$ is multiplicatively dependent with $p$ (otherwise we may apply   Laurent's theorem  as before). 
Also, we may multiply everywhere by a power $p^{An}$. So, moving again along each among  finitely many arithmetical progressions of $n$, we may write
\begin{equation*}
r_i=p^{b_i},\qquad 0\le b_1<b_2<\ldots <b_s,
\end{equation*}
for integers $b_i$. Note that if some $b_i$ were negative, then equation \eqref{u in T} becomes trivial.

On changing slightly the notation, we may write $k_1=k_2=1$  and assume that for our solutions  we have $0\le n_1\le n_2$ (and that $n_1,n_2$ are divisible by certain fixed integers $k_1,k_2$).   We may also suppose that $u_n\neq 0$. 

Let $v=v_p$ be the $p$-adic order function, extended in some way to a number field containing the relevant quantities.

The equation gives   $v(c_1p^{n_1}+c_2p^{n_2})\ge b_1n+O(1)$, so we may assume that $n_1,n_2\ge b_1n+O(1)$. Then, dividing out by $p^{b_1n}$ we reduce to the case $b_1=0$. 

Now, $v(P_1(n)-c_1p^{n_1}-c_2p^{n_2})\ge b_2n+O(1)$.  If  $P_1(n)-c_1p^{n_1}-c_2p^{n_2}\neq 0$ we deduce successively that $n_2\ge b_2n+O(1)$ and then $n_1\ge b_2n+O(1)$, which would imply  $v(P_1(n))\gg n$, which  is inconsistent for large $n$. 

Therefore 
\begin{equation} \label{E.eq}
P_1(n)=c_1p^{n_1}+c_2p^{n_2}.
\end{equation}
Implicitly this implies $s=1$, and anyway this is the equation we shall deal with from now on.

For a given solution, we shall write 
\begin{equation*}
q:=c_1p^{n_1}+c_2p^{n_2},
\end{equation*}
 and we shall assume that $n$ is large enough to justify the subsequent approximations. So, $q$ shall also be large and therefore $n_2=\max\{n_1,n_2\}$ shall be large as well.  

We write $d:=\deg P_1>0$.

We split the analysis into two similar cases, according as $n_1<n_2/2$ or not.

{\bf Case I}. Suppose first $n_1<n_2/2$, so $\ell:=n_2-n_1>n_2/2$.  In equation \eqref{E.eq}, we expand in a Puiseux series for $n$ at $\infty$ in terms of $q$:
\begin{equation}\label{E.pui}
n=q^{1/d}\sum_{r=0}^\infty \gamma_rq^{-r/d},
\end{equation}
where $\gamma_r$ are certain coefficients depending only on $P_1$. Of course we tacitly mean that some definite branch has been chosen for the $d$-th root. 

In turn, we may  write $q^{-r/d}=c_2^{-r/d}p^{-rn_2/d}\left(1+\frac{c_1}{c_2}p^{-\ell}\right)^{-r/d}$ and expand the factor on the right by the binomial theorem.  Doing that for all $r\ge  0$ in \eqref{E.pui} we may write
\begin{equation}\label{E.pui2}
n=p^{n_2/d}f(p^{-n_2/d},p^{-\ell}),
\end{equation}
where $f(x,y):=\sum_{a,b\ge 0}\sigma_{a,b}x^ay^b$ is a certain power series with   coefficients in a  suitable  number field, converging for $x,y$ in a  region  $\max\{|x|,|y|\}\le \rho$ for a sufficiently small $\rho>0$. 

We apply to this series and solutions Theorem 1 of the paper \cite{CZ} (written by two of the authors), on taking therein $K$ a number field containing all relevant numbers, $S$ the set of  places of $K$ which are either archimedean or lie above $p$, and the sequence of points $\x :=(x_1,x_2)=\left(p^{-n_2/d},p^{-\ell}\right)$.  

In the Case I that we are treating we also choose $\nu$ as a complex absolute value for which the identity \eqref{E.pui2} holds.

Note that $f(\x)=np^{-n_2/d}$  is an $S$-integer, so assumption (3) for  \cite[Theorem~1]{CZ} holds. Also, the coordinates of $\x$ are $S$-units, so another of the assumptions in \cite[Theorem~1]{CZ} indeed holds. Further, we have $h(\x)\asymp n_2\asymp -\log |x_1|_\nu\asymp -\log |x_2|_\nu$, and this takes care of all other required inequalities.

The conclusion of the theorem delivers finitely many algebraic cosets $\u_1H_1,\ldots ,\u_\ell H_\ell$ of $\G_m^2$, whose union contains all of our points $\x$,  and such that the restriction of $f$ to each $\u_iH_i$ is a polynomial. Also, the identities \eqref{E.pui2} express the integers $n$ in term of these polynomials and points.

Suppose first that one of the cosets is the whole $\G_m^2$. Then $f$ is a polynomial in its arguments, and \eqref{E.pui2} says it satisfies $P_1(x^{-1}f(x,y))=(c_1y+c_2)x^{-d}$. This entails that $d=1$, so $P_1(n)=an+b$ and $af(x,y)=c_1y+c_2-bx$.  This yields the required shape for the solutions $n$.

Suppose now that each of the cosets is proper. We can focus on the cosets of dimension $1$ and containing infinitely many elements $\x$, the others giving rise to finitely many solutions. Since $\x$ tends to zero the coset may be   parametrized by $x=t^a, y=\mu t^b$, with coprime  integers $a,b>0$ and a nonzero constant $\mu$. The restriction of $f$ to this coset corresponds to $f(t^a,\mu t^b)$, which is thus a polynomial, denoted $\phi(t)$. We also have
\begin{equation*}
P_1(t^{-a}\phi(t))=c_1\mu t^{b-da}+c_2t^{-da}.
\end{equation*}
We see that $\phi(0)\neq 0$ and if $c:=\deg \phi$ we have $b=dc$.  Thus the left side is a Laurent-polynomial in  $t^d$, and is thus invariant by $t\mapsto \zeta t$ with any $d$-th root of unity $\zeta$. We could exploit this fact through Galois groups, but we choose a more direct way to describe the situation completely.

Our solutions are obtained by substitutions $t= p^{-n_2/(da)}$ and we have $$da\ell=bn_2\text{, }n_2-\frac{bn_2}{da}=n_1\ge 0,$$ 
so $b<da$.  Setting $1/t$ in place of $t$ and changing slightly notation, we may write
\begin{equation}\label{E.P_1}
P_1(\psi(t))=c_1t^{du}+c_2t^{da},
\end{equation}
where $0\le u=a -c<a$ and where $\psi(t)=t^a\phi(t^{-1})$. Note that $\psi(t)$ is necessarily a polynomial, of degree $a$.

If $d=1$ our conclusion follows at once, so suppose $d>1$. If $\rho$ is a root of $P_1$ of multiplicity $\mu>1$ then necessarily $\psi(t)-\rho$ equals a constant times $t^{du}$, since the right side of \eqref{E.P_1} has only $t=0$ as a multiple root. This again yields the required representations for $n$.

Suppose then that $P_1$ has no multiple roots. Still, $t^{du}$ shall divide exactly one factor $\psi(t)-\rho$, so $du\le \deg \psi=a$. 

Let now $Q(x)=P(x+\beta)$ where $\beta$ is chosen so that $Q$ has vanishing second coefficient, and let $\xi(t)=\psi(t)-\beta$, so $Q(\xi(t))=P_1(\psi(t))$.  For $\delta$ the leading coefficient of $Q$, we have 
\begin{equation}
\label{former inequality}
\deg(\delta\xi(t)^d-c_2t^{da})\le \max\{d\ell, (d-2)\deg\xi\}\le \max \{a, (d-2)a\}.
\end{equation}

However $\delta\xi(t)^d-c_2t^{da}$ is a product of $d$ terms of the shape $\gamma_1\xi(t)-\gamma_2 \zeta t^a$, where $\zeta$ runs over the $d$-th roots of unity, and at most  one term has degree $<a$. Hence either $\delta\xi(t)^d-c_2t^{da}$ vanishes, or it has degree $\ge (d-1)a$, and in  case of equality one of the factors has to be constant. 

Comparing with the inequality \eqref{former inequality}, in any case again we find that $\psi(t)$ is of the shape $\gamma t^a+\gamma'$, and we conclude as before.  

{\bf Case II}. We are left with the case when $n_1\ge n_2/2$. Now we argue in an entirely similar way, but expanding for $n$ at $q=0$ and using a $p$-adic absolute value in place of an archimedean one; however a few supplementary arguments are needed because when $n_1/n_2$ is very near to $1$ the Puiseux expansion may not  obey  the needed conditions. To dispose of this difficulty we argue as follows.

If $n_2-n_1$ is bounded for infinitely many solutions, then we fall essentially into the equation $P_1(n)=c_3p^{n_1}$; this is   standardly known, and may be solved even effectively for given $P_1,p$ (there are only finitely many solutions unless $P_1$ is a power of a linear polynomial). So we assume that $n_2-n_1$ tends to infinity for our solutions. 

First, $P_1(x)$ may be supposed clearly with rational coefficients, and let us consider  powers $Q(x)^m$ dividing it exactly, where $Q$ is an irreducible polynomial in $\Q[x]$ and $m\ge 1$.  Now, one such  factor at least will be such that $v(Q(n)^m)=n_1+O(1)$, while we must have bounded $p$-adic order for the other factors. 

Suppose first that there are no other non constant factors, so $P_1=cQ^m$;  then $c_1+c_2p^{n_2-n_1}$ must be  an $m$-th power, up to a factor taken from a finite set. By well-known results this leads to only finitely many solutions if $m>1$, so let us assume $m=1$. In this case there is no ramification above $q=0$, so we may expand the solutions of $P_1(n)=q$ at $q=0$ in a power series $n=\sum_{r=0}^\infty \gamma_rq^r$ converging in any absolute value for which $q$ is small enough.\footnote{Several power series may appear, corresponding to the various branches, represented by the roots of $P_1=0$.}  
We choose a $p$-adic place $\nu$ of a number field $K$ containing all the coefficients $\gamma_r$. Note that since we have no ramification, the powers of $q$ which appear are integers, and we may expand $q^r=(c_1p^{n_1}+c_2p^{n_2})^r$ with the binomial theorem, to obtain a series $f(x,y)=\sum_{a,b\ge 0}\sigma_{a,b}x^ay^b$, with $n=f(p^{n_1},p^{n_2})$ in the valuation $\nu$. Now the proof proceeds exactly as in Case I (actually even with the slight simplification of having no fractional powers involved). (This case actually falls as part of a result proved in \cite{CZ1}.)

Suppose now that there is some other non-constant  factor. Then, since the product of the other factors contributes at most essentially $p^{n_2-n_1}$,  we find  $p^{n_2-n_1}\gg n$, whereas $p^{n_2}\ll n^d$, so $n_1\le n_2(1-d^{-1})+O(1)$.  With this inequality, we can use a Puiseux expansion at $q=0$, even if there is ramification. Corresponding to each branch, we have an expansion $n=\sum_{r\ge 0}\gamma_rq^{r/e}$ for a suitable integer $e>0$, converging at any $p$-adic place $\nu$ for small enough $|q|_\nu$, i.e. for large enough $n_1$. We can further expand the powers $q^{r/e}$  as $q^{r/e}=c_1^{r/e}p^{rn_1/e}(1+(c_2/c_1)p^{n_2-n_1})^{r/e}$, using the binomial theorem on the right. The fact that $n_2-n_1\gg n_2$ enables us to obtain a sufficiently fast convergence in order to  apply again Theorem 1 
 of \cite{CZ} as in Case I. This  shall lead to the sought conclusion as above.
\end{proof}

\begin{remark}
It will be noted that for $n_1>\epsilon n_2$ (for any fixed $\epsilon >0$), we can use either argument for both cases.  Under such type of inequalities actually the arguments work for the representations of $u_n$ as the sum of any  prescribed number of terms of the shape $c_ip^{n_i}$. It is the lack of such inequalities that prevent the proof to go through for  the general case; so already for three terms one meets problems if $n_1/n_3$ tends to $0$ and $n_2/n_3$ tends to $1$ through a sequence of solutions. See \cite{CZ3} for devices to overcome this kind of difficulty in special cases. 

Also, the expansions used in the proof might be explained in geometric terms, after suitable blow-ups.
\end{remark}


\section{Proof of Theorems~\ref{surface result} and \ref{Frobenius result}}
\label{section proofs}

Using Theorem~\ref{theorem polynomial-exponential}, we deduce our main results.

\begin{proof}[Proof of Theorem~\ref{surface result}.]
Using Theorem~\ref{theorem reduction}, we know that the set 
$$S=\{n\in\N_0\colon \Phi^n(\alpha)\in V\}$$
(where $\Phi:\bG_m^N\lra \bG_m^N$ is a regular self-map and $V\subset \bG_m^N$ is a surface)  
is a finite union of generalized $F$-arithmetic sequences. Since the minimal polynomial $P_{{\rm min},F}\in\Z[x]$ of the Frobenius has a unique root equal to $p$, then each such generalized $F$-arithmetic sequence is  the intersection of finitely many $F$-arithmetic sequences, each one of them consisting of all non-negative integers $n$ which belong to a suitable arithmetic progression and furthermore, for which there exist some $m\in\N_0$ and some $n_1,\dots, n_m\in\N_0$ such that
\begin{equation}
\label{F sequence 1}
u_n=c_1p^{k_1n_1}+\cdots + c_mp^{k_mn_m},
\end{equation}
for some given linear recurrence sequence $\{u_n\}$, some given integers $c_i$ and some given non-negative integers $k_i$. Furthermore, by Theorem~\ref{theorem reduction}~(1), since $\dim(V)=2$, we have that at most two of the $k_i$'s in equation \eqref{F sequence 1} are nonzero. So, at the expense of replacing $u_n$ by $u_n-c$ (for a suitable integer), we may assume $m=2$. Then Theorem~\ref{theorem polynomial-exponential}~(B) yields that each such $F$-arithmetic sequence is a union of finitely many arithmetic progressions and finitely many $p$-sets of the form \eqref{set 2 p}. Then Propositions~\ref{prop:p-arithmetic and arithmetic}~and~\ref{prop:intersections of p-arithmetic sequences} yield that these finitely many intersections of $F$-arithmetic sequences are once again a finite union of arithmetic progressions and finitely many $p$-sets.
\end{proof}

\begin{proof}[Proof of Theorem~\ref{Frobenius result}.]
Again using Theorem~\ref{theorem reduction}, we obtain that the set $S$ consisting of all $n\in\N_0$ such that $\Phi^n(\alpha)\in V$ is a finite union of generalized $F$-arithmetic sequences, and furthermore, each such generalized $F$-arithmetic sequence is a finite intersection of finitely many $F$-arithmetic sequences, each one of them consisting of all non-negative integers $n$ belonging to a suitable arithmetic progression and furthermore, for which there exist $n_1,\dots,n_m\in\N_0$ such that
\begin{equation}
\label{F sequence 2}
u_n=c_1p^{k_1n_1}+\cdots + c_mp^{k_mn_n},
\end{equation}
for some given linear recurrence sequence $\{u_n\}$, some given $m\in\N_0$, some given integers $c_i$ and some given non-negative integers $k_i$. As shown in Theorem~\ref{theorem reduction}~(2), since $\Phi:\bG_m^N\lra\bG_m^N$ is a group endomorphism, the characteristic roots of the linear recurrence sequence $\{u_n\}$ from \eqref{F sequence 2} are positive integer powers of the roots of the minimal polynomial $P_{{\rm min},\Phi}\in\Z[x]$ of $\Phi$. Now, the hypothesis regarding $\Phi$ from Theorem~\ref{Frobenius result} yields that no root of $P_{{\rm min},\Phi}$ is  multiplicatively dependent with the prime $p$; then Theorem~\ref{theorem polynomial-exponential}~(A) finishes the proof of our result. Indeed, if $P_{{\rm min},\Phi}$ were to have a root $\lambda$ which is multiplicatively dependent with the prime $p$, then we would have that for some iterate $\Phi^r$ (for a suitable $r\in\N$) there is a root of its minimal polynomial $P_{{\rm min},\Phi^r}\in\Z[x]$ which equals $p^s$ for some $s\in\N_0$. Letting $G=\ker(\Phi^r-p^s{\rm Id})$, where ${\rm Id}$ is the identity morphism on $\bG_m^N$, we obtain that $G$ is a positive dimensional algebraic subgroup of $\bG_m^N$ and the induced action of $\Phi$ on $G$ is given by a power of the Frobenius, contradiction.

This concludes the proof of Theorem~\ref{Frobenius result}.
\end{proof}



\begin{thebibliography}{GTZ9999}
\newcommand{\au}[1]{{#1},}
\newcommand{\ti}[1]{\textit{#1},}
\newcommand{\jo}[1]{{#1}}
\newcommand{\vo}[1]{\textbf{#1}}
\newcommand{\yr}[1]{(#1),}
\newcommand{\page}[1]{#1.}
\newcommand{\ppx}[1]{#1,}
\newcommand{\pps}[1]{#1.}
\newcommand{\bk}[1]{{#1},}
\newcommand{\inbk}[1]{in: {#1}}
\newcommand{\xxx}[1]{{arXiv:#1.}}




\bibitem[BG06]{BG06}
\au{E.~Bombieri and W.~Gubler}
\ti{Heights in diophantine geometry}
New Mathematical Monographs, vol.~4, Cambridge University Press, Cambridge, 2006.



\bibitem[BGT16]{book}
\au{J.~P.~Bell, D.~Ghioca, and T.~J.~Tucker}
\ti{The Dynamical Mordell-Lang Conjecture}
Mathematical Surveys and Monographs \textbf{210}, American Mathematical Society, Providence, RI, 2016, xiv+280 pp.


\bibitem[CZ00]{CZ1}
P.~Corvaja and U.~Zannier, \emph{On the Diophantine equation $f(a^m,y)=b^n$},  Acta Arith. \textbf{94} (2000), no.~1, 25-–40. 

\bibitem[CZ05]{CZ} 
P. Corvaja and  U. Zannier, \emph{$S$-unit points on analytic hypersurfaces},  Ann. Sci. \`Ec. Norm. Sup.  \textbf{38} (2005), 76--92.

\bibitem[CZ13]{CZ3}
P.~Corvaja and U.~Zannier, \emph{Finiteness of odd perfect powers with four nonzero binary digits}, Ann. Inst. Fourier (Grenoble) \textbf{63} (2013), no.~2, 715–-731.



\bibitem[Der07]{Derksen}
H.~Derksen, \emph{A Skolem-Mahler-Lech theorem in positive characteristic and finite automata}, Invent. Math. \textbf{168} (2007), no.~1, 175-–224.

\bibitem[DM12]{Derksen-Masser}
H.~Derksen and D.~Masser, \emph{Linear equations over multiplicative groups, recurrences, and mixing I}, Proc. Lond. Math. Soc. (3) \textbf{104} (2012), no.~5, 1045-–1083. 

\bibitem[Eis95]{Eisenbud}
D.~Eisenbud, \emph{Commutative algebra: With a view toward algebraic geometry},
Graduate Texts in Mathematics {\bf 150}, Springer-Verlag, New York, 1995.
      
	
\bibitem[Fal94]{Fal94}
\au{G.~Faltings} 
\ti{The general case of {S.}~{L}ang's conjecture}
Barsotti Symposium in Algebraic Geometry (Albano Terme, 1991),
Perspective in Math., vol.~15, Academic Press, San Diego, 1994, pp.~175--182 

\bibitem[Ghi08]{groups}
D.~Ghioca, \emph{The isotrivial case in the Mordell-Lang theorem}, Trans. Amer. Math. Soc. \textbf{360} (2008), no.~7, 3839–-3856.

\bibitem[Ghi]{G-TAMS}
D.~Ghioca, \emph{The dynamical Mordell-Lang conjecture in positive characteristic}, Trans. Amer. Math. Soc., to appear (2018), 18 pp.  


\bibitem[GT09]{GT-JNT}
\au{D.~Ghioca and T.~J.~Tucker} 
\ti{Periodic points, linearizing maps, and the dynamical Mordell-Lang problem}
\jo{J. Number Theory}
\vo{129}
\yr{2009}
\pps{1392--1403}



\bibitem[GTZ11]{GTZ-Bordeaux}
 \au{D.~Ghioca, T.~J.~Tucker, and M.~E.~Zieve}
\ti{The Mordell-Lang question for endomorphisms of semiabelian varieties}
\jo{J. Theor. Nombres Bordeaux}
\vo{23}
\yr{2011}
\pps{645--666}

\bibitem[Hru96]{Hrushovski}
E.~Hrushovski, \emph{The Mordell-Lang conjecture for function fields}, J. Amer. Math. Soc. \textbf{9} (1996), no.~3, 667–-690.





\bibitem[Lau84]{Laurent}
M.~Laurent, \emph{\'{E}quations diophantiennes exponentielles}, Invent. Math. \textbf{78} (1984), no.~2, 299-–327.



\bibitem[Lec53]{Lech}
\au{C.~Lech}
\ti{A note on recurring series}
\jo{Ark. Mat.}
\vo{2}
\yr{1953}
\pps{417--421}

\bibitem[Mah56]{Mahler}
K.~Mahler, \emph{On the Taylor coefficients of rational functions}, 
Proc. Cambridge Philos. Soc. \textbf{52} (1956), 39–-48.

\bibitem[Mas04]{Masser}
D.~Masser, \emph{Mixing and linear equations over groups in positive characteristic}, Israel
J. Math. \textbf{142} (2004), 189--204.

\bibitem[Mil]{Milne}
J.~Milne, \emph{Abelian varieties}, lecture notes available online.

\bibitem[MS04]{Moosa-Scanlon}
R.~Moosa and T.~Scanlon, \emph{$F$-structures and integral points on semiabelian varieties over finite fields}, Amer. Journal of Math. \textbf{126} (2004),  473--522.

\bibitem[Nel17]{Ne17}
\au{K. Nelson}
\ti{Two special cases of the dynamical {M}ordell-{L}ang
conjecture}
Master's thesis, University of British Columbia, March 2017.



\bibitem[SY14]{SY2014}
\au{T.~Scanlon and Y.~Yasufuku}
\ti{Exponential-polynomial equations and dynamical return sets}
\jo{Int. Math. Res. Not. IMRN}
\vo{2014}
\yr{2014}
\pps{4357--4367}


\bibitem[Sch03]{Sch03}
\au{W.~Schmidt}
\ti{Linear recurrence sequences}
Diophantine Approximation (Cetraro, Italy, 2000),
Lecture Notes in Math. 1819, Springer-Verlag 
Berlin Heidelberg, 2003, pp.~171--247.



	
\bibitem[Sko34]{Skolem}
\au{T.~Skolem}
\ti{Ein Verfahren zur Behandlung gewisser exponentialer Gleichungen und diophantischer
Gleichungen}
Comptes Rendus Congr. Math. Scand. (Stockholm, 1934) 163--188.



\bibitem[Voj96]{Voj96}
\au{P.~Vojta}
\ti{Integral points on subvarieties of semiabelian varieties, {I}}
\jo{Invent. Math.}
\vo{126}
\yr{1996}
\pps{133--181}

\bibitem[Zan09]{Zannier}
U.~Zannier, \emph{Lecture notes on Diophantine analysis. 
With an appendix by Francesco Amoroso}, Scuola Normale Superiore di Pisa (Nuova Serie) \textbf{8}. Edizioni della Normale, Pisa, 2009. xvi+237 pp.

\end{thebibliography}
\end{document}